\documentclass[a4paper,10pt]{article}
\usepackage{amsmath}
\usepackage{hyperref}
\usepackage{amsfonts}
\usepackage{amssymb}
\usepackage{amsthm}
\usepackage{graphics}
\usepackage{color}
\usepackage[pdftex]{graphicx}
\newtheorem{theorem}{Theorem}[section]
\newtheorem{lemma}[theorem]{Lemma}
\newtheorem{corollary}[theorem]{Corollary}
\newtheorem{definition}[theorem]{Definition}
\newtheorem{example}[theorem]{Example}
\def\bee{\begin{equation}}
\def\eee{\end{equation}}

\title{\large{\textbf{ON A UNIFIED THEORY OF NUMBERS}}}
\author{
Nilotpal Kanti Sinha
\\
\small{Great Lakes Institute of Management, Chennai, India.}
\\
\small{e-mail: \texttt{nilotpalsinha@gmail.com}}
\\
\\
Marek Wolf
\\
\small{Institute of Theoretical Physics,Wroclaw, Poland.}
\\
\small{e-mail: \texttt{mwolf@ift.uni.wroc.pl}}}

\begin{document}
\maketitle{\textit{Dedicated to Grigori Yakovlevich Perelman.}}

\begin{abstract}
\small{We show that several results in seemingly independent areas of number theory such as divergent series, summation of arithmetic functions, uniform distribution modulo one and summation over prime numbers can be unified under a single equation. We introduce the concept of asymptotic equidistribution and apply our method to derive several new results in the above areas of topics, especially in the theory of prime numbers.}
\end{abstract}

\section{Introduction}

In 1873 James Clerk Maxwell combined electricity and magnetism using the Maxwell's Equations into a single force called electromagnetism and since then, physicists have been trying to discover the elusive unified theory that can explain all the forces in nature using a common set of equations. Although the title of our paper may appear too broad at this stage, we present the analogous idea of unifying seemingly unrelated topics in number theory such as the divergent series, summation of arithmetic functions, uniform distribution modulo one and summation involving prime numbers, under one common equation.
\\
\\
We shall see that several results in the above areas of mathematics which are currently treated as independent results can be unified under a  single equation using summations of the type $\sum_{r=1}^{n} d_r f(S_r)$ where
$S_r = \sum_{i=1}^{r} d_i$. In section 3, we present our unified Master Theorem. In section 4, we present the Equivalence Theorem which shows that the ratios of the partial sum $\frac{S_r}{S_n}$ of a divergent series has properties analogous to that of a sequence $a_n$ uniformly distributed modulo one. In section 5, we introduce the concept of asymptotic equidistribution modulo one and prove the sequence of prime numbers is asymptotic equidistributed modulo one. In section 6 and 7, we apply the Master Theorem and its variations to derive several new results on summations involving prime numbers and derive beautiful asymptotic formula such as:
\\
\begin{displaymath}
\lim_{n \rightarrow \infty}\frac{1}{\ln^2 n}\sum_{r=1}^{n} \frac{1}{\gamma_r}\Big(\frac{1}{r}+\frac{1}{p_r}\Big)
\Big(\tan^{-1} \frac{\gamma_r}{\gamma_n}\Big)^2 e^{H_r + \frac{1}{p_1} + \ldots + \frac{1}{p_r}}
= \frac{e^{\gamma + M}}{4}\Big(G - \frac{7\zeta(3)}{4\pi}\Big)
\end{displaymath}
\\
which connects Apery's constant $\zeta(3)$, Catalan constant G, Euler-Mascheroni constant $\gamma$, exponential constant $e$, harmonic numner $H_n$, Meissel-Merten constant $M$, $n^{th}$ prime number $p_n$, imaginary part of the $n^{th}$ non-trivial zero of the Riemann zeta function $\gamma_n$ and the fundamental constants of a circle $\pi$. In section 9, we present the generalized the definitions of the Riemann zeta function and the Euler constants. All the identities that we have discovered in this paper through rigorous methods have also been numerically verified, thereby confirming our theory of unification. We have calculated the constants that appear in different equations by direct computation using prime numbers up to $2^{44}$. In section 10, we give the details of the verifications of identities as well as computations of constants. To conclude our paper, in section 11, we give an outline of some of the work that can be done in future in this line of research.
\\
\\
And with due apologies to Abel, we shall see in this paper divergent series are not a useless invention of the devil after all. Divergent series act as the bond that unifies different areas of mathematics.

\section{Preliminary lemmas}

In the rest of this paper, $a_n$ shall denote a sequence uniformly distributed modulo one. For the sake of simplicity, we assume that $0<a_n<1$. If $a_n>1$ then its fractional part \{$a_n$\} is to be used.

\begin{lemma}
If $a_n$ is uniformly distributed modulo one then for every Riemann integrable function $f$ in $[0,1]$,
\begin{equation}
\lim_{n \to \infty} \frac {1}{n} \sum_{r=1}^{n} f(a_r) = \int_{0}^{1}f(x)dx.
\label{lemma2.1a}
\end{equation}
\end{lemma}

\begin{proof}
Well known. (See \cite{gr}, Page 3)
\qedhere\
\end{proof}

\begin{lemma}
If f is monotonic and continuous and defined in $[1,n]$ then,
\begin{equation}
\sum_{r=1}^{n} f(r)= \int_{1}^{n} f(x)dx + O(|f(n)|+|f(1)|).
\label{lemma2.2a}
\end{equation}
Further if f is continuously decreasing to zero then,
\begin{equation}
\sum_{r=1}^{n} f(r)= \int_{1}^{n} f(x)dx + C_f + O(f(n)).
\label{lemma2.2b}
\end{equation}
where $C_f$ is a constant that depends only on f.
\end{lemma}

\begin{proof}
Well known. (See \cite{ik}, 1.62-1.67, Page 19-20)
\qedhere\
\end{proof}

\begin{lemma}
Let $S(n)=\sum_{r=1}^{n} {b(r)}$. If $g$ is defined in $[S(1),S(n)]$ and $b(x)g(S(x))$ is monotonic and continuous in $[1,n]$ then,
\begin{equation}
\sum_{r=1}^{n} b(r)g(S(r)) = \int_{S(1)}^{S(n)}g(x)dx + O(\delta (n))
\label{lemma2.3a}
\end{equation}
where $\delta (n) = |b(n)g(S(n)|+|b(1)g(S(1))| + \int_{1}^{n}b'(y)g(S(y))dy$.
\end{lemma}

\begin{proof} 
From \ref{lemma2.2a}, we have
\\
\begin{displaymath}
S(y) = \sum_{r=1}^{y}b(r) = \int_{1}^{y} b(t)dt + O(|b(y)|+|b(1)|).
\end{displaymath}
\\
Differentiating $S(y)$ we get $S'(y)=b(y)+O(b'(y))$. Let $S(y)=x$; then,
\begin{displaymath} S'(y)dy = b(y)dy + O(b'(y)dy) = dx \end{displaymath}
\begin{displaymath} b(y)dy = dx + O(b'(y)dy) \end{displaymath}
\begin{displaymath} g(S(y))b(y)dy=g(x)dx+O(b'(y)g(S(y))dy) \end{displaymath}
\begin{equation}
\int_{1}^{n}b(y)g(S(y))dy = \int_{S(1)}^{S(n)}g(x)dx + O \Big(\int_{1}^{n}b'(y)g(S(y))dy \Big)
\label{lemma2.3b}
\end{equation}
\\
Again, from \ref{lemma2.2a}, we have
\begin{equation} 
\sum_{r=1}^{n}b(r)g(S(r)) = \int_{1}^{n}b(y)g(S(y))dy +O(\delta (n))
\label{lemma2.3c}
\end{equation}
where $O(\delta (n)) = |b(n)g(S(n)|+|b(1)g(S(1))| $.
\\
\\
Using \ref{lemma2.3b}, we can express \ref{lemma2.3c} as
\begin{displaymath}
\sum_{r=1}^{n}b(r)g(S(r)) = \int_{S(1)}^{S(n)}g(x)dx + O \Big(\delta (n) + \int_{1}^{n}b'(y)g(S(y))dy \Big).
\end{displaymath}
\\
The proof is completed by redefining the error term $\delta (n)$ as
\begin{displaymath}
\delta (n) = |b(n)g(S(n)|+|b(1)g(S(1))| + \int_{1}^{n}b'(y)g(S(y))dy.
\qedhere\
\end{displaymath}
\end{proof}

Notice that \ref{lemma2.3a} is a generalization of \ref{lemma2.2a}. If $b(n)=1$ then, in the error term $\delta(n)$, the integral vanishes because $b'(t)=0$ and so \ref{lemma2.3a} reduces to \ref{lemma2.2a}. In this paper, we are interested in unifying different areas of number theory rather than on developing a general summation formula. So improving the error term $\delta(n)$ is not in scope of our current paper.

\begin{lemma} 
If $a_n $ is uniformly distributed modulo m and $f(n)$ is divergent then as $n \rightarrow \infty$,
\begin{equation}
\sum_{r=1}^{n}a_r f(r) = (1+o(1))\frac{m}{2}\sum_{r=1}^{n}f(r)
\label{lemma2.4}
\end{equation}
\end{lemma}

\begin{proof} 
Since $a_n$ is uniformly distributed modulo $m$, $(m-a_n )$ is also uniformly distributed modulo $m$. Hence as
$n \rightarrow \infty$,
\\
\begin{displaymath}
\sum_{r=1}^{n}a_r f(r) \sim \sum_{r=1}^{n}(m - a_r)f(r) = m \sum_{r=1}^{n}f(r) - \sum_{r=1}^{n}a_r f(r),
\end{displaymath}
\begin{displaymath}
\sum_{r=1}^{n}a_r f(r) \sim \frac{m}{2} \sum_{r=1}^{n}f(r)
\end{displaymath}
\\

Therefore if $f(n)$ is divergent and  $E(n)$ is an error term such that
\\
\begin{displaymath}
\sum_{r=1}^{n}a_r f(r) = \frac{m}{2} \sum_{r=1}^{n}f(r) +E(n)
\end{displaymath}
\\
then $\lim_{n \rightarrow \infty} \frac{E(n)}{\sum_{r=1}^{n}f(r)} = 0$. Therefore $E(n)=o(\sum_{r=1}^{n}f(r))$. This completes the proof of the lemma.
\qedhere\
\end{proof}

\begin{lemma} 
If $a_n$ is uniformly distributed modulo one and $g(n)$ is a monotonic and continuous divergent series of positive terms then, as $n \rightarrow \infty $,
\begin{equation}
\sum_{r=1}^{n} f(a_r)g(r) = (1+o(1))\int_{0}^{1}f(x)dx \sum_{r=1}^{n}g(r).
\label{lemma2.5a}
\end{equation}
\end{lemma}

\begin{proof}
The sum $\sum_{r=1}^{n}f(a_r )g(r)$  will depend on the sequence $a_r$, $n$, the function $f$ and $g$, and may be also on some other hidden function say $h$ which is independent of $f$ and $g$. Hence without loss of generality, we can express the above sum as

\begin{equation}
\sum_{r=1}^{n}f(a_r )g(r) = F(n)G(n)H(n) + E(n)
\label{lemma2.5b}
\end{equation}
\\
where $F(n)$ depends only on $f$, $G(n)$ depends only on $g$ and $H(n)$ includes all terms which are not contributed by  $f$ and $g$; this may include constants and other terms (if any); and $E(n)$ is an unknown error term. If suppose any one or more terms of the factorization $F(n)G(n)H(n)$ did not exist then the corresponding value of $F(n)$ or $G(n)$ or $H(n)$ will be unity. Later we shall see that $H(n)=1$; but at this stage we do not know this.
\\

Our approach will be in two steps. In the first step, we choose the function $g$ and impose no restriction on $f$. Since there is no restriction on $f$ hence in the resulting sum, $F(n)$ and $H(n)$ will be the general form of $F$ and $H$ respectively while $G(n)$ will be special case of $G$ for the chosen function $g$.
\\

In the second step, we choose the function $f$ and impose no restriction on $g$. In this case, $G(n)$ and $H(n)$ will be the general form of $G$ and $H$ respectively while $F(n)$ will be a special case of $F$ for the chosen function $f$. Since by definition $F$, $G$ and $H$ are independent of each other, therefore after these two steps we will get the general forms of $F(n)$ and $G(n)$.
\\
\\
\textbf{Step 1}. Taking $g(x)=1$ in \ref{lemma2.5b} we get

\begin{displaymath}
\sum_{r=1}^{n}f(a_r ) = F(n)G_1(n)H(n) + E_1(n)
\end{displaymath}
\\
where $G_1 (n)$ and $E_1 (n)$ denote the function $G(n)$ and $E(n)$ respectively in the special case when $g(x)=1$. From 
\ref{lemma2.1a}, $E_1 (n)=o(n)$. Hence,

\begin{equation}
F(n)H(n) = \frac{1}{G_1(n)}\Bigg(\sum_{r=1}^{n}f(a_r) - o(n)\Bigg).
\label{lemma2.5c}
\end{equation}
\\
\textbf{Step 2}. Taking $f(x)=x$ in \ref{lemma2.5b}, we get

\begin{displaymath}
\sum_{r=1}^{n}a_r g(r) = F_1(n)G(n)H(n) + E_2(n)
\end{displaymath}
\\
where $F_1 (n)$ and $E_2 (n)$ denote the function $F(n)$ and $E(n)$ respectively in the special case when $f(x)=x$. From 
\ref{lemma2.4}, $E_2 (n)=o(\sum_{r=1}^{n}g(r))$. Also since $a_n$ is uniformly distributed modulo one, ($1-a_n )$ is also uniformly distributed modulo one and therefore, as $n \rightarrow \infty$,

\begin{displaymath}
\sum_{r=1}^{n}(1-a_r)g(r) = F_1(n)G(n)H(n) + o\Big(\sum_{r=1}^{n}g(r) \Big),
\end{displaymath}
\begin{displaymath}
\sum_{r=1}^{n}g(r) = 2F_1(n)G(n)H(n) + o\Big(\sum_{r=1}^{n}g(r) \Big),
\end{displaymath}
\begin{equation}
G(n)H(n) = \frac{1}{2F_1(n)}\Bigg(\sum_{r=1}^{n}g(r) - o\Big(\sum_{r=1}^{n}g(r) \Big)\Bigg).
\label{lemma2.5d}
\end{equation}

From \ref{lemma2.5c}; with $g(x)=1$ and $f(x)$ any arbitrary function; and \ref{lemma2.5d}; with $f(x)=x$, $g(x)$ any arbitrary function; we have
\begin{equation}
F(n)G(n)H(n)^{2} = \frac{(\sum_{r=1}^{n}f(a_r) - o(n))(\sum_{r=1}^{n}g(r) - o(\sum_{r=1}^{n}g(r))}{2F_1(n)G_1(n)}.
\label{lemma2.5e}
\end{equation}

Multiplying the terms in RHS of \ref{lemma2.5e} and simplifying using Lemma 2.1, we obtain

\begin{equation}
F(n)G(n)H(n)^{2} = (1+o(1))\int_{0}^{1}f(x)dx \sum_{r=1}^{n}g(r).
\label{lemma2.5f}
\end{equation}

Notice that dominant term of in the RHS of \ref{lemma2.5f} equation can be completely evaluated if we know $f$ and $g$. Therefore there is no hidden term $H(n)$ or we can simply say that $H(n)=1$. This completes the proof of the lemma.
\qedhere\
\end{proof}

\section{The Master Theorem}
With the above lemmas, we are now ready to unify divergent series, summation of arithmetic functions and uniform distribution modulo one into a single relation which we call the Master Theorem.

\begin{theorem} 
If $a_n$ is uniformly distributed modulo one and $d_n$ is positive and the following conditions are satisfied:
\\
\\
$(i)$   $g$ is Riemann integrable in $[S_1,S_n ]$ where $S_n = \sum_{r=1}^{n}d_r$,
\\
$(ii)$  $d_x g(S_x)$ is monotonic and continuous in $[1,n]$,
\\
\\
then for every function $f$ Riemann integrable in $[0,1]$,
\\
\begin{equation}
\sum_{r=1}^{n}d_r g(S_r)f(a_r) \sim \int_{S_1}^{S_n}g(x)dx \int_{0}^{1}f(x)dx.
\label{theorem3.1a}
\end{equation}
\end{theorem}

\begin{proof} 
The proof follows directly from \ref{lemma2.5a} by replacing $g(r)$ with $d_r g(S_r )$ and then using 
\ref{lemma2.3a}. The replacement is justified because $d_r g(S_r )$ is a function of $r$ and is independent of $f(a_r )$.
\qedhere\
\end{proof}

Notice all the lemmas of Section 2 are now unified under the Master Theorem. We shall now apply the Master Theorem to derive some of the existing as well as new results in different areas of mathematics. In the rest of this paper the symbols and notations introduced in the Master Theorem shall hold good. Also, we shall assume that the conditions mentioned in the Master Theorem are satisfied.

\begin{example}
\begin{equation}
\lim_{n \rightarrow \infty} \frac{1}{S_n^{2}} \sum_{r=1}^{n}a_r d_r S_r = \frac{1}{4},
\label{example3a}
\end{equation}
\begin{equation}
\lim_{n \rightarrow \infty} \frac{1}{n} \sum_{r=1}^{n}\frac{e^{H_r}}{r(1+a_r^{2})} = \frac{\pi e^{\gamma}}{4},
\label{example3b}
\end{equation}
\begin{equation}
\lim_{n \rightarrow \infty} \frac{1}{n^{2}} \sum_{r < n}\frac{n-r}{(\ln n - \ln r)(1+a_r)} = (\ln 2)^{2}.
\label{example3c}
\end{equation}
\end{example}

\section{The Equivalence Theorem}

Although the Master Theorem unifies divergent series with uniform distribution modulo one, we have a more explicit unification between the two. We shall prove that the ratios of the partial sums of a divergent series replicate the behaviors of sequence uniformly distributed modulo one. With the same notations and definitions as in \ref{theorem3.1a}, we have:

\begin{theorem} 
If $d_n$ is positive and divergent and $\lim_{n \rightarrow \infty}\frac{S_{n-1}}{S_n} = 1$ then,
\begin{equation}
\lim_{n \rightarrow \infty}\frac{1}{S_n} \sum_{r=1}^{n}d_r f \Big(\frac{S_r}{S_n}\Big) = \int_{0}^{1}f(x)dx.
\label{theorem4.1a}
\end{equation}
\end{theorem}

\begin{proof} 
Proceeding exactly as in the proof of \ref{lemma2.3a}, we can show that
\begin{equation}
\sum_{r=1}^{n} d_r f \Big(\frac{S_r}{S_n}\Big) = \int_{S_1}^{S_n}f \Big(\frac{y}{S_n}\Big)dy
+O\Big(|d_n f(1)|+ \Big|d_1 f\Big(\frac{S_1}{S_n}\Big)\Big|\Big).
\label{theorem4.1b}
\end{equation}

Substituting $x = y/S_n$ and making $n \rightarrow \infty$, \ref{theorem4.1b} reduces to
\begin{equation}
\sum_{r=1}^{n} d_r f \Big(\frac{S_r}{S_n}\Big) = S_n \int_{0}^{1}f(x)dx +O(|d_n f(1)|+|d_1 f(0)|).
\label{theorem4.1c}
\end{equation}

Since $S_n$ is divergent, dividing both sides of \ref{theorem4.1c} by $S_n$, we obtain
\begin{displaymath}
\lim_{n \rightarrow \infty}\frac{1}{S_n} \sum_{r=1}^{n}d_r f \Big(\frac{S_r}{S_n}\Big) = \int_{0}^{1}f(x)dx.
\qedhere\
\end{displaymath}
\end{proof}
 
Thus we have explicitly unified divergent series with uniform distribution modulo one. It must be noted that above theorem cannot have a general extension analogous to the Master Theorem because the sequence of the ratios of the partial sums  $\frac{S_r}{S_n}$  are not necessarily uniformly distributed modulo one.

\begin{example}
\begin{equation}
\lim_{n \rightarrow \infty} \sum_{r < n} \frac{d_r}{S_n\sqrt{\ln S_n - \ln S_r}} = \sqrt{\pi}.
\end{equation}
\end{example}

\section{Asymptotic equidistribution modulo one}

A standard application of definite integrals that is taught in elementary calculus in evaluating limit of sums using the rectangle method

\begin{equation}
\lim_{n \rightarrow \infty}\frac{1}{n} \sum_{r=1}^{n}f \Big(\frac{r}{n}\Big) = \int_{0}^{1}f(x)dx.
\label{paragraph5a}
\end{equation}

This formula is similar to \ref{lemma2.1a}. If we take $d_r=1$ in Theorem 4.1 we get both \ref{lemma2.1a} and \ref{paragraph5a}. Thus the Equivalence Theorem not only explains why these two formulas are similar but also shows that they are actually two different forms of the same principle. Also since \ref{paragraph5a} holds for all functions $f$  Riemann integrable in [0,1] it implies that the sequence

\begin{displaymath}
\frac{1}{n},\frac{2}{n},...,\frac{n}{n}
\end{displaymath}
approaches uniform distribution modulo one as $n \rightarrow \infty$. Inspired form this example, we introduce the concept of asymptotic equidistribution modulo one.

\begin{definition}
A sequence of numbers $b_n$ is said to be asymptotically equidistributed modulo one if the sequence of ratios
\begin{displaymath}
\frac{b_1}{b_n},\frac{b_2}{b_n},...,\frac{b_n}{b_n}
\end{displaymath}
approach uniform distribution modulo one as $n \rightarrow \infty$.
\end{definition}

A consequence of this definition, the following corollary follows trivially from the Equivalence theorem \ref{theorem4.1a} that the following corollary holds:

\begin{corollary} 
If $b_n$ is asymptotically equidistributed modulo one  and $S_n = \sum_{r=1}^{n}b_r$ then,
\begin{displaymath}
\lim_{n \rightarrow \infty} \frac{1}{S_n}\sum_{r=1}^{n}b_r f\Big(\frac{S_r}{S_n}\Big) 
= \lim_{n \rightarrow \infty} \frac{1}{n}\sum_{r=1}^{n}f\Big(\frac{b_r}{b_n}\Big)
= \int_{0}^{1}f(x)dx.
\end{displaymath}
\end{corollary}

This naturally brings the question: Which are the sequences that are asymptotically equidistributed modulo one. In this regard, we have the following lemma.

\begin{lemma} 
If $b_n$ is a sequence such that $\lim_{n \rightarrow \infty}\frac{b_{[nt]}}{b_n} = t$ for every $t, 0<t<1$, then $b_n$ is asymptotically equidistributed modulo one.
\label{lemma5.3}
\end{lemma}

\begin{proof}
Let $0 \le a < b \le 1$. If the condition mentioned in the statement of the lemma is true then,
\\
\begin{displaymath}
\lim_{n \rightarrow \infty}\frac{b_{[bt]} - b_{[at]}}{b_n} = b-a.
\end{displaymath}
\\
Therefore $n \rightarrow \infty$ the probability of finding an integer $r$ such that $a \le b_r/b_n \le b$ approaches  $b-a$ and hence $b_n$ is asymptotically equidistributed modulo one.
\qedhere\
\end{proof}

\begin{corollary} 
If two sequences are asymptotically equidistributed modulo one then their linear combinations are also asymptotically equidistributed modulo one.
\end{corollary}

\begin{proof}
Trivial.
\qedhere\
\end{proof}

Using this criteria, we can easily prove that the sequence $b_n=n^c$ is asymptotically equidistributed modulo one only if $c=1$. Similarly we can prove that the sequence of harmonic numbers $H_n$ is not asymptotically equidistributed modulo one. A very useful application of Lemma \ref{lemma5.3} is in the sequence of primes. We have:

\begin{theorem}
The sequence of primes is asymptotically equidistributed modulo one.
\end{theorem}

\begin{proof} 
From the prime number theorem, $\lim_{n \rightarrow \infty}\frac{p_n}{n \ln n} = 1$. Hence, if $0<t<1$ then,
\\
\begin{displaymath}
\lim_{n \rightarrow \infty}\frac{p_{[tn]}}{p_n} = t + \lim_{n \rightarrow \infty}\frac{t \ln t}{\ln n} = t.
\end{displaymath}
\\
Therefore by Lemma \ref{lemma5.3} the sequence of primes is asymptotically equidistributed modulo one.
\qedhere\
\end{proof}

Similarly we show from the asymptotic expansion of the $n^{th}$ composite number (See \cite{ab}) that the sequence of composite numbers, $c_n$, is asymptotically equidistributed modulo one and consequently, we have:

\begin{lemma} 
If $p_n$ and $c_n$ denote the sequence of prime numbers and the sequence of composite numbers respectively and $\alpha, \beta$ and $\gamma$ are constants, such that $\alpha p_n + \beta c_n + \gamma n \ne 0$ then,
\begin{displaymath}
\lim_{n \rightarrow \infty}\frac{1}{n} \sum_{r=1}^{n} f \Big(\frac{\alpha p_r + \beta c_r + \gamma r}
{\alpha p_n + \beta c_n + \gamma n}\Big) = \int_{0}^{1}f(x)dx.
\end{displaymath}
\end{lemma}

\begin{proof} 
Since the sequence $p_n,c_n$ and $n$ all satisfy the condition of Lemma \ref{lemma5.3}, they are all asymptotically equidistributed modulo one and hence their linear combination will also be asymptotically equidistributed modulo one.
\qedhere\
\end{proof}

\begin{example}
\begin{equation}
\lim_{x \rightarrow \infty} \frac{1}{\pi (x)} \sum_{p \le x}\Big(\frac{p}{x}\Big)^{a-1} \ln^{b-1} \Big(\frac{cx}{p}\Big)
= \frac{c^a}{a^b}\Gamma (b, a\ln c)
\label{example5.7a}
\end{equation}

\begin{equation}
\sum_{r=1}^{n}{p_r}^{a}({p_1}^{a}+{p_2}^{a} + \ldots + {p_r}^{a})^{b}
\sim \frac{n^{b+1} p_n^{ab+a}}{(b+1)(a+1)^{b+1}}
\label{example5.7b}
\end{equation}

\begin{equation}
\lim_{n \rightarrow \infty} \frac{\sum_{r=1}^{n}r^{ab+a+b}}{\sum_{r=1}^{n}r^a(1^a + 2^a + \ldots + r^a)^b}
= (a+1)^b
\end{equation}

\begin{equation}
\lim_{n \rightarrow \infty} \frac{1}{n} \sum_{r=1}^{n}\frac{p_n ^{2}}{p_n ^{2}+ p_r ^{2}}
=\lim_{n \rightarrow \infty} \frac{1}{n} \sum_{r=1}^{n}\frac{c_n ^{2}}{c_n ^{2}+ c_r ^{2}}
=\lim_{n \rightarrow \infty} \frac{1}{n} \sum_{r=1}^{n}\frac{n ^{2}}{n ^{2}+ r ^{2}} = \frac{\pi}{4}
\end{equation}
where $\Gamma(u,v)$ denotes the incomplete gamma function in $\ref{example5.7a}$ and $a \ge 1$ in $\ref{example5.7b}$.
\end{example}

\section{Summation over prime numbers}

By defining $d_r$ as a characteristic function that vanishes whenever $r$ does not have a predefined characteristic, we can adapt the Master theorem to evaluate summation over prime numbers by defining $d_r = 0$ when $r$ is composite. With this adaptation we have:

\begin{lemma}
If $p_n$ is the $n^{th}$ prime number then,
\begin{equation}
\sum_{r=1}^{n}d_{p_r} g(S_{p_r})f(a_r) \sim \int_{S_2}^{S_{p_r}}g(x)dx \int_{0}^{1}f(x)dx.
\label{lemma6.1}
\end{equation}
\end{lemma}

\begin{proof}
The proof follows directly by replacing $d_r$ with $d_{p_r}$ in the Master Theorem \ref{theorem3.1a}.
\qedhere\
\end{proof}

\begin{lemma} 
If $p_n$ is the $n^{th}$ prime number then,
\begin{equation}
\sum_{r=1}^{n}d_{p_r} g(S_{p_r}) = \int_{S_2}^{S_{p_n}}g(y)dy + O(\delta (n)).
\label{lemma6.2}
\end{equation}
where $O(\delta (n)$ is defined analogous to the definition given in $\ref{lemma2.3a}$.
\end{lemma}

\begin{proof}
The proof follows directly from \ref{lemma2.3a}.
\qedhere\
\end{proof}

\begin{example} 
If $\theta(x)$ denotes the Chebyshev function of the first kind then
\begin{equation}
\sum_{p\le x} \frac{\ln p}{\theta (p)} = \ln x + O(1).
\label{example6.3}
\end{equation}
\end{example}

\begin{proof}
Form the Prime Number Theorem we have $\lim_{x \rightarrow \infty} \frac{\theta (x)}{x}=1$. Hence, taking $d_{p_r} = \ln p$, and $g(x) = \frac{1}{x}$ in \ref{lemma2.3a}, we have

\begin{displaymath}
\sum_{p\le x} \frac{\ln p}{\theta (p)} = \ln \theta (x) + O(1) = \ln x +O(1).
\end{displaymath}
\\
By actual computation, we have found that

\begin{displaymath}
\lim_{x \rightarrow \infty}\Big(\sum_{p\le x} \frac{\ln p}{\theta (p)} - \ln x\Big)
\end{displaymath}
\\
converges to a constant value $\theta \approx 0.50904\ldots$. The details of the calculation of the constant $\theta$ are given in section 10.
\qedhere\
\end{proof}

\begin{example} 
Similarly we can show that if $\psi(x)$ denotes the Chebyshev function of the second kind then
\begin{equation}
\sum_{p^{v}\le x} \frac{\ln p}{\psi (p)} = \ln x + O(1).
\label{example6.4}
\end{equation}
\end{example}

\section{Consequences of Merten's Theorem}

One of the most celebrated result in the theory of prime numbers is the Merten's Theorem on the sum of the reciprocal of primes;

\begin{equation}
\sum_{p\le x} \frac{1}{p} = \ln\ln x + M + O\Big(\frac{1}{\ln x}\Big).
\label{paragraph7a}
\end{equation}
\\
where $M \approx 0.2614972\ldots$ is the Meissel-Mertens constant(See \cite{hw}, Theorem 429). In this section, we shall apply the Master Theorem on Merten's Theorem to derive several beautiful identities on prime numbers. Taking $d_{p_r} = 1/p_r$ in \ref{lemma6.2}, we have

\begin{lemma}
If f is monotonically decreasing then
\begin{equation}
\sum_{r \le x}\frac{1}{p_r}f\Big(\frac{1}{p_1}+ \frac{1}{p_1} + \ldots + \frac{1}{p_r}\Big)
\sim \int_{0.5}^{\ln \ln x + M}f(x)dx
\label{lemma7.1}
\end{equation}
\end{lemma}

\begin{proof}
Taking $d_{p_r} = 1/p_r$ in \ref{lemma6.2} and simplifying using \ref{paragraph7a} we obtain the required result.
\qedhere\
\end{proof}

\begin{corollary}
If $c>1$ and $m =[\ln p_n]$ then
\begin{equation}
\sum_{r=1}^{m}\frac{c^{M+H_r}}{r}f(a_r) \sim \sum_{r=1}^{n}\frac{c^{\gamma+Q_r}}{p_r}f(a_r)
\sim \frac{c^{\gamma +M}}{\ln c}\int_{0}^{1}f(x)dx
\label{corollry7.3}
\end{equation}
\\
where $H_r = 1 + \frac{1}{2} + \ldots{} +\frac{1}{r}$ and $Q_r = \frac{1}{p_1} + \frac{1}{p_2} + \ldots{} + \frac{1}{p_r}$.
\end{corollary}
 
\begin{example} 
Taking $c=e$, $f(x)=1$ and $m =[\ln p_n]$ in $\ref{corollry7.3}$ we obtain
\\
\begin{equation}
\sum_{r=1}^{m} \frac{e^{M + 1 + \frac{1}{2} + \ldots{} +\frac{1}{r}}}{r}
= \sum_{r=1}^{n} \frac{e^{\gamma + \frac{1}{p_1} + \frac{1}{p_2} +\ldots{}+ \frac{1}{p_r}}}{p_r}
+ H + O\Big(\frac{1}{\ln p_n}\Big)
\label{example7.4}
\end{equation}
\\
where $H$ is a constant. Calculation of the constant $H$ is given in Section 9.
\end{example}
 
\begin{example}
Taking $d_{p_r}= 1/r + 1/p_r$ in $\ref{lemma6.1}$, we obtain the identity; given in the introduction in section 1; that connects all the four fundamental constants $e,\pi , \gamma$ and $M$.
\begin{equation}
\lim_{n \rightarrow \infty}\frac{1}{p_n}\sum_{r=1}^{n}\Big(\frac{1}{r}+\frac{1}{p_r}\Big)\frac{e^{G_r}}{1+{a_r}^{2}}
= \frac{\pi e^{\gamma + M}}{4}
\label{example7.5}
\end{equation}
\\
where $G_r = 1 + \frac{1}{2} + \ldots{} +\frac{1}{r} + \frac{1}{p_1} + \frac{1}{p_2} +\ldots{}+ \frac{1}{p_r}$.
\end{example}

\begin{example}
Taking $c=e$ and $f(x)=1$ in $\ref{corollry7.3}$ and comparing the result with Merten's theorem
$\prod_{p \le x}\Big(1-\frac{1}{p}\Big) \sim \frac{e^{-\gamma}}{\ln x}$, we obtain
\begin{equation}
\prod_{r=1}^{n}\Big(\frac{p_r}{p_r - 1}\Big)
= \sum_{r=1}^{n} \frac{e^{\gamma -M + \frac{1}{p_1} + \frac{1}{p_2} +\ldots{}+ \frac{1}{p_r}}}{p_r} + O(1).
\label{example7.6}
\end{equation}
\end{example}

The above identity is interesting because it connects a product involving the first $n$ primes to a sum involving the 
first $n$ primes. By actual calculation, we found that $O$-constant which we denote by $\rho \approx 0.76774\ldots$.

\section{A beautiful formula}

In this section, we shall demonstrate how our idea of unifying different area of mathematics can be used to yields beautiful results. As an example, we shall derive the intriguing formula that we have seen in the introduction in section 1.
\\
\begin{displaymath}
\frac{1}{\ln^2 n}\sum_{r=1}^{n} \frac{1}{\gamma_r}\Big(\frac{1}{r}+\frac{1}{p_r}\Big)
\Big(\tan^{-1} \frac{\gamma_r}{\gamma_n}\Big)^2 \exp\Big(H_r + \frac{1}{p_1} + \frac{1}{p_2} + \ldots + \frac{1}{p_r}\Big)
\end{displaymath}
\begin{equation}
= \frac{e^{\gamma + M}}{4}\Big(G - \frac{7\zeta(3)}{4\pi}\Big) + O\Big(\frac{\ln\ln n}{\ln n}\Big)
\end{equation}

This formula brings together the elements from nine different topics of number theory into a single beautiful result. We have the Apery's constant $\zeta(3)$, Catalan constant G, Euler-Mascheroni constant $\gamma$, exponential constant $e$, harmonic numner $H_n$, Meissel-Merten constant $M$, $n^{th}$ prime number $p_n$, imaginary part of the $n^{th}$ non-trivial zero of the Riemann zeta function $\gamma_n$ and the fundamental constants of a circle $\pi$.
\\

Since $\gamma_n \sim \frac{2\pi}{\ln n}$, $\gamma_n$ is asymptotically equidistributed modulo one. Taking 
\\
\begin{displaymath}
d_r = \frac{1}{r} + \frac{1}{p_r}, a_r = \frac{\gamma_r}{\gamma_n}, f(x) = \frac{(\tan^{-1}x)^2}{x}
\end{displaymath}
\\
in the Master Theorem, the result follows.

\section{Generalized zeta function and Euler constants}

Throughout this paper, we have used summation of the type $\sum_{r=1}^{n}d_r f(S_r )$ where  $S_r=\sum_{k=1}^{r}d_k$ . This suggests that many mathematical properties could be unearthed by studying series of this type. Although there are several generalizations of the Riemann zeta series and the Euler constants, in this section, we present a new generalization of each of these two terminologies.

\begin{definition}
The weighted zeta function is defined as
\\
\begin{equation}
\zeta_{d_r}(s) = \sum_{n=1}^{\infty}\frac{d_n}{S_n ^{s}}.
\label{paragraph8a}
\end{equation}
\end{definition}

Thus Riemann zeta series (See~\cite{et}) is the simplest of the family of weighted zeta series, corresponding to the special case with unit weights $d_n=1$. The Riemann zeta series is related to prime numbers through the Euler product; it would be interesting to investigate the analogues of prime numbers for the weighted zeta series, if there are any.

\begin{definition}
If  $d_rf(S_r )$  is positive, strictly decreasing and divergent then the generalized Euler constants are defined as
\\
\begin{equation}
\gamma _{d_x,f(x),g(x)} = \lim_{n \rightarrow \infty}\Big(\sum_{r=1}^{n}d_r f(a_r) g(S_r)
- \int_{0}^{1}f(x)dx \int_{S_1}^{S_n}g(x)dx\Big).
\label{paragraph8b}
\end{equation}
\end{definition}

The classical Euler-Mascheroni, $\gamma \approx 0.577215665\ldots$, corresponds to the special case when $d_x=f(x)=1$ and $g(x)=1/x$.

\begin{example}
\begin{displaymath}
\gamma _{1/x,1,x^a} = \gamma(a) = \lim_{n \rightarrow \infty}\Big(\sum_{r=1}^{n}\frac{H_r^a}{r} - \frac{H_n ^{a+1}}{a+1}\Big).
\end{displaymath}
\end{example}
With this definition, we have $\gamma(0) = \gamma$, $\gamma(1) \approx 0.8225\ldots.$

\section{Computations}

We have performed some computer calculation to determine the $O(1)$ terms in some of the formulas as well as to check 
accuracy of some of them. All programs were written in Intel$^{\small{\textregistered}}$ Fortran and run on the 64 bits AMD$^{\small{\textregistered}}$ Opteron 2700 MHz processors.

\subsection{Constant of \ref{example6.3}}

We will present the data for determination of $O(1)$ term in:
\[
\lim_{x \rightarrow \infty }\Big(\sum_{p\le x} \frac{\ln p}{\theta (p)} - \ln x\Big).
\]
The results are presented in the Table II; the values of $x$ are in powers of 10 and powers of 2, the largest is 
$ x=2^{46}=7.037\ldots \times 10^{13}$. We infer from this Table, that in \ref{example6.3}  $\theta \approx 0.50904...$.

\vskip -1.7cm
\begin{center}
{\sf TABLE {\bf I}}\\
\bigskip
\begin{tabular}{||c|c|c||} \hline
$x$ & $ \sum_{p\le x} \frac{\ln p}{\theta (p)} $ & $ \sum_{p\le x} \frac{\ln p}{\theta (p)} - \ln(x)$ \\ \hline
           10000000  &   15.6266542966473 &     0.5085586456\\ \hline
           16777216  &   16.1442710909375 &     0.5087387574\\ \hline
       \vdots   &   \vdots &  \vdots  \\ \hline
     10000000000000  &   29.4426443525143 &     0.5090381435\\ \hline
     17592186044416  &   30.0075139149438 &     0.5090379703\\ \hline
     35184372088832  &   30.7006623741728 &     0.5090392489\\ \hline
     70368744177664  &   31.3938126082563 &     0.5090423025\\ \hline
\end{tabular} \\
\end{center}
\vskip 0.4cm
Again values of $x$ formed  the progressions $x=10^n$ and $x=2^n$ and last value is $x=2^{46}=7.04\times 10^{13}$ 
and it took 1 month CPU time  to reach this value.

\subsection{Constant of \ref{example6.4}}

It is not easy to estimate the big-O constant  appearing in \ref{example6.4} because the Chebyshev function $\psi(x)$ 
for each $x$ has to be calculated separately, i.e. it is not possible to use the value $\psi(x)$ to calculate
$\psi(y)$  for $y>x$. It is clearly  seen from the following formula:
\bee
\psi(x) = \sum_{p^k\le x}\log (p)=\sum_{n \leq x} \Lambda(n) = \sum_{p\le x}\lfloor\log(x)/\log(p)\rfloor\log(p)
\label{psi_x}
\eee
where $\Lambda(n)$ is of course the Mangoldt function. In the last form in (\ref{psi_x}) the summand evidently depends on $x$.
Due to this obstacle we have calculated
\[
\sum_{p<x} \frac{\ln(p)}{\psi(p)} - \ln(x)
\]
only for arithmetic progression $x=n1000000$, instead of geometric progressions in the previous cases. The obtained data is 
presented in the Table IV.

\vskip 0.4cm
\begin{center}
{\sf TABLE {\bf II}}\\
\bigskip
\begin{tabular}{||c|c|c||} \hline
$x$ & $\sum_{p<x} \frac{\ln(p)}{\psi(p)} $ & $ \sum_{p<x} \frac{\ln(p)}{\psi(p)} - \ln(x)$ \\ \hline
            1000000  &   13.5176919008455 &    -0.2978186571188\\ \hline
            2000000  &   14.2109984371948 &    -0.2976593013294\\ \hline
            3000000  &   14.6162538296483 &    -0.2978690169841\\ \hline
            4000000  &   14.9037288388882 &    -0.2980760801960\\ \hline
            \vdots   &   \vdots    &  \vdots \\ \hline
          120000000  &   18.3046100971201 &    -0.2983922036263\\ \hline
          121000000  &   18.3128899746922 &    -0.2984111288688\\ \hline
          122000000  &   18.3211283777276 &    -0.2984032249699\\ \hline
          123000000  &   18.3292829091801 &    -0.2984120041566\\ \hline
          124000000  &   18.3374008408318 &    -0.2983912827375\\ \hline
          125000000  &   18.3454350973536 &    -0.2983891979130\\ \hline
          126000000  &   18.3534102043258 &    -0.2983822605899\\ \hline
\end{tabular} \\
\end{center}
\vskip 0.4cm
Above results suggest that in \ref{example6.4}:
\[
\sum_{p^{v}\le x} \frac{\ln p}{\psi (p)}= \ln x +O(1)
\]
there will be no equality, as LHS contains more terms and will be larger than:
\[
\sum_{p\le x} \frac{\ln p}{\psi (p)}
\]
Indeed the computer has produced data:

\vskip 0.4cm
\begin{center}
{\sf TABLE {\bf III}}\\
\bigskip
\begin{tabular}{||c|c|c||} \hline
$x$ & $ \sum_{p^{v}\le x} \frac{\ln p}{\psi (p)} $ & $ \sum_{p^{v}\le x} \frac{\ln p}{\psi (p)}-\ln(x)$ \\ \hline
            1000000  &   50.9115451893350 &    37.0960346313707\\ \hline
            2000000  &   54.6640936160706 &    40.1554358775464\\ \hline
            3000000  &   56.6781935153063 &    41.7640706686740\\ \hline
            4000000  &   57.2884765277029 &    42.0866716086188\\ \hline
            5000000  &   59.5704206070202 &    44.1454721366218\\ \hline
            6000000  &   60.2669838560583 &    44.6597138288660\\ \hline
            7000000  &   60.6473757815097 &    44.8859550744901\\ \hline
            8000000  &   61.0151186141244 &    45.1201665144803\\ \hline
            9000000  &   62.1874985975655 &    46.1747634622650\\ \hline
           10000000  &   62.7528494571091 &    46.6347538061507\\ \hline
          \vdots   &  \vdots & \vdots  \\  \hline
           26000000  &   66.9185401883590 &    49.8449330923732\\ \hline
           27000000  &   66.9728912056416 &    49.8615437816730\\ \hline
           28000000  &   67.0256045468440 &    49.8778894787045\\ \hline
           29000000  &   67.2571282847139 &    50.0743218967631\\ \hline
           30000000  &   67.3132959212713 &    50.0965879816448\\ \hline
           31000000  &   67.3999358663197 &    50.1504381038703\\ \hline

\end{tabular} \\
\end{center}
\vskip 0.4cm

\subsection{Constant of \ref{example7.4}}

We tried to determine the constant hidden in $O(1)$ in \ref{example7.4}. The results are not conclusive as the difference
\[
\sum_{r=1}^{m} \frac{e^{M + 1 + \frac{1}{2} + \ldots{} +\frac{1}{r}}}{r}
- \sum_{r=1}^{n} \frac{e^{\gamma + \frac{1}{p_1} + \frac{1}{p_2} +\ldots{}+ \frac{1}{p_r}}}{p_r}
\]
displays fluctuations between 5 and 7, see Table III

\vskip -1.4cm
\begin{center}
{\sf TABLE {\bf IV}}\\
\bigskip
\begin{tabular}{||c|c|c|c||} \hline
$x$ & $ \sum_{r=1}^{m} \frac{e^{M + 1 + \frac{1}{2} + \ldots{} +\frac{1}{r}}}{r}$ & $ \sum_{r=1}^{n} \frac{e^{\gamma + \frac{1}{p_1} + \frac{1}{p_2} +\ldots{}+ \frac{1}{p_r}}}{p_r}$  & difference\\ \hline
           10000000  &     41.0329563695293 &    34.3855942478502 &      6.6473621216792\\ \hline
           16777216  &     41.0329563695293 &    35.5826945591639 &      5.4502618103654\\ \hline
           33554432  &     43.4147084333527 &    37.1863081123819 &      6.2284003209708\\ \hline
           67108864  &     45.7926459440599 &    38.7895728672373 &      7.0030730768227\\ \hline
               \vdots  & \vdots  & \vdots & \vdots \\ \hline
         100000000000  &     62.3588796562188 &    55.6923315736851 &      6.6665480825336\\ \hline
       137438953472  &     62.3588796562188 &    56.4280087636232 &      5.9308708925955\\ \hline
       274877906944  &     64.7168951615670 &    58.0315248369373 &      6.6853703246296\\ \hline
       549755813888  &     67.0732528789821 &    59.6350423017193 &      7.4382105772628\\ \hline
      1000000000000  &     67.0732528789821 &    61.0190970811536 &      6.0541557978285\\ \hline
      1099511627776  &     67.0732528789821 &    61.2385589160106 &      5.8346939629716\\ \hline
      2199023255552  &     69.4280715529841 &    62.8420759971757 &      6.5859955558084\\ \hline
      4398046511104  &     71.7814576116916 &    64.4455941070075 &      7.3358635046840\\ \hline
      8796093022208  &     71.7814576116916 &    66.0491118126049 &      5.7323457990866\\ \hline
     10000000000000  &     71.7814576116916 &    66.3458668229100 &      5.4355907887816\\ \hline
     17592186044416  &     74.1335068131232 &    67.6526285355308 &      6.4808782775924\\ \hline
     35184372088832  &     76.4843056253357 &    69.2561434465050 &      7.2281621788307\\ \hline

\end{tabular} \\
\end{center}
\vskip 0.4cm

\subsection{Verification of \ref{example7.5}}

We shall discuss the formula

\begin{displaymath}
\lim_{n \rightarrow \infty}\frac{1}{p_n}\sum_{r=1}^{n}\Big(\frac{1}{r}+\frac{1}{p_r}\Big)\frac{e^{G_r}}{1+{a_r}^{2}}
= \frac{\pi e^{\gamma + M}}{4}=1.8169017889\ldots,
\end{displaymath}

\begin{displaymath}
G_r = 1+ \frac{1}{2} + \ldots{} +\frac{1}{r} + \frac{1}{p_2} +\ldots{}+ \frac{1}{p_r}
\end{displaymath}
\\
which connects the fundamental mathematical constants $\pi, e, \gamma, M$ and contains also sequence $a_r$ uniformly 
distributed on the interval $(0,1)$. The numbers $a_r$ were generated from the sequence

\[
Z_{r+1} = A Z_r~ {\rm mod }~C, ~~~~a_r=\frac{Z_r}{C}
\]
\\
where  $A=1203248318$ and $C=2^{31} -1$ (Mersenne prime). It is known, that period of $Z_n$  is C and the sequence is 
uniformly distributed over $(0,1)$  \cite{Fishman-1982}. The period of $a_r$ was closed during calculations many times 
as there are $346065536839>2^{38}$ primes up to $10^{13}$,  but we believe it is not an obstacle in this problem.  The 
results are presented in the Table I; the values of $x$ form the progression $10^n$ and $2^n$ --- the last value is for
$x=2^{46}=7.037\ldots \times 10^{13}$.

\vskip 0.4cm
\begin{center}
{\sf TABLE {\bf V}}\\
\bigskip
\begin{tabular}{||c|c|c||} \hline
$n$ & $ \sum_{r=1}^{n}\Big(\frac{1}{r}+\frac{1}{p_r}\Big)\frac{e^{G_r}}{1+{a_r}^{2}} $ & $ \sum_{r=1}^{n}\Big(\frac{1}{r}+\frac{1}{p_r}\Big)\frac{e^{G_r}}{1+{a_r}^{2}}- \frac{\pi e^{\gamma + M}}{4}$ \\ \hline
    1.000000000 $\times 10^{  9}$  &      1.914474310 &     0.097544058\\ \hline
    4.294967291 $\times 10^{  9}$  &      1.907421278 &     0.090491026\\ \hline
   1.0000000000 $\times 10^{ 10}$  &      1.903773984 &     0.086843732\\ \hline
   1.7179869143 $\times 10^{ 10}$  &      1.901578829 &     0.084648577\\ \hline
      \vdots   &   \vdots &  \vdots  \\ \hline
 1.000000000000 $\times 10^{ 12}$  &      1.888054174 &     0.071123922\\ \hline
 1.099511627689 $\times 10^{ 12}$  &      1.887789938 &     0.070859685\\ \hline
 2.199023255531 $\times 10^{ 12}$  &      1.885918614 &     0.068988361\\ \hline
 4.398046511093 $\times 10^{ 12}$  &      1.884144386 &     0.067214134\\ \hline
 8.796093022151 $\times 10^{ 12}$  &      1.882458766 &     0.065528514\\ \hline
1.0000000000000 $\times 10^{ 13}$  &      1.882156156 &     0.065225904\\ \hline
1.7592186044399 $\times 10^{ 13}$  &      1.880856104 &     0.063925852\\ \hline
3.5184372088777 $\times 10^{ 13}$  &      1.879329708 &     0.062399456\\ \hline
7.0368744177643 $\times 10^{ 13}$  &      1.877873176 &     0.060942923\\ \hline
\end{tabular} \\
\end{center}
\vskip 0.4cm

\subsection{Constant of \ref{example7.6}}

Writing \ref{example7.6} in the form $L(n)=e^{\gamma-M+o(1)}R(n)+O(1)$, where:

\bee
L(n)=\prod_{k=1}^n \frac{p_k}{p_k-1}
\eee
and $R(n)$ is RHS:
\bee
R(n)=\sum_{r=1}^n \frac{1}{p_r}\exp{\left(\sum_{k=1}^r \frac{1}{p_k}\right)}.
\eee
\\
we see, that $\rho$ can be found by plotting points $(L(n), R(n))$ for some values of $n$ and fitting the straight line to 
these points by  the least square method.  We computed the values of $(L(n), R(n))$ on the computer and stored to the file 
at $n=10^6, 2^{24}, 2^{25}, \ldots, 10^{13}, 2^{46}=  7.037...\times 10^{13}$. The figure 1 presents obtained points.
\\
\begin{figure}
\vspace{-3.0cm}
\hspace{-3.5cm}
\begin{center}
\includegraphics[height=0.5\textheight,angle=0]{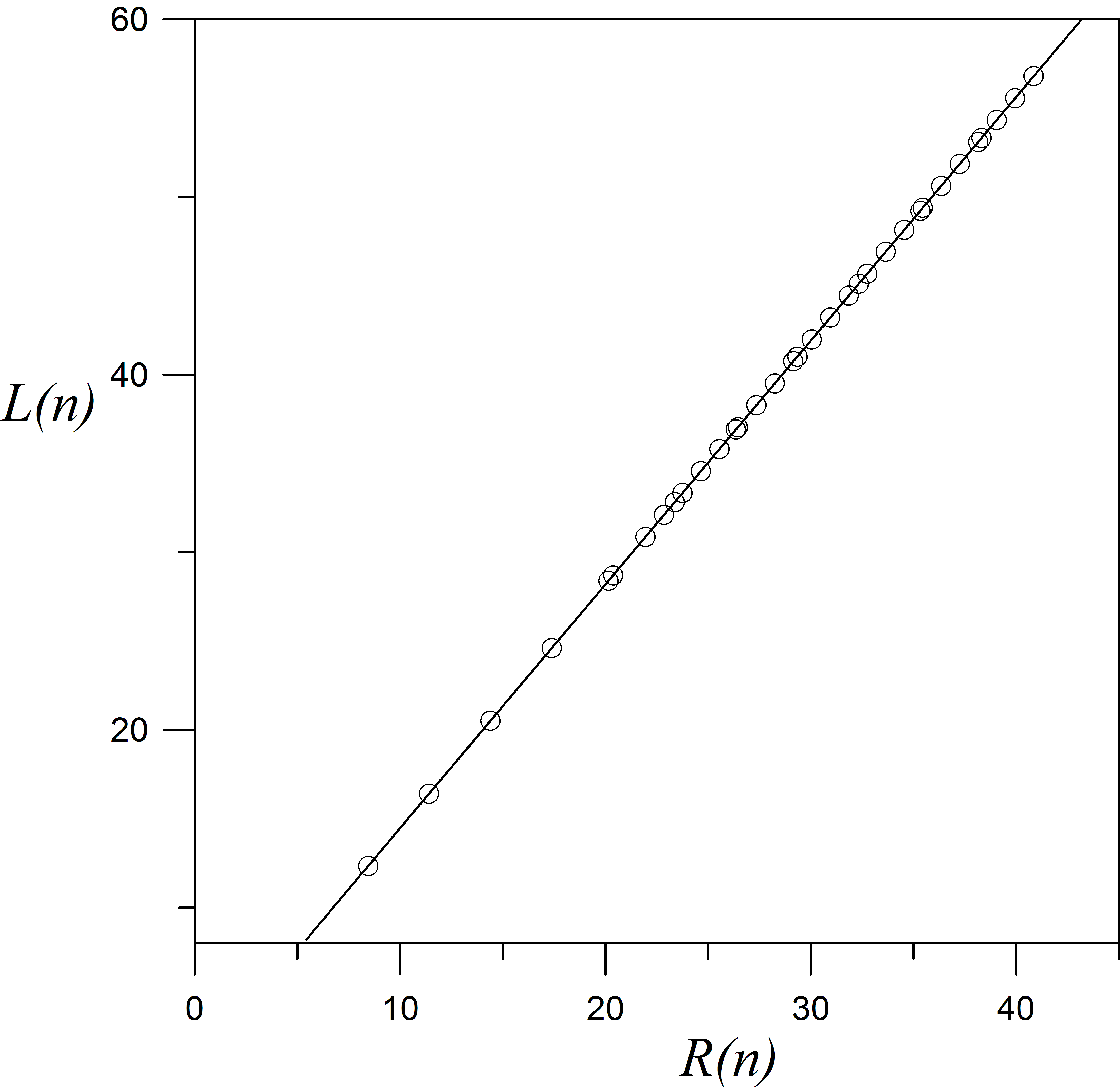} \\
\end{center}
\end{figure}
\\
Fitting the straight line to the last 4 points (in fact two points determine line!)
gives $L = 1.371247673 R + 0.7676022888$,
while $e^{\gamma-M}=1.3712441303$, thus 6 digits of the correct value are
reproduced. Assuming that for $\rho$  also 6 digits are correct we have $\rho=0.76770$.
Inspection of the Sloane's Online Encyclopedia of Integer Sequences does not
indicate what this number can be in terms of other mathematical constants.

\newpage

\section{Future works}
The scope for future works should focus on extending the master theorem to other areas including the complex domain. In particular since uniform distribution is a special case of low discrepancy sequences, we would want to extend the results of this paper to low discrepancy sequences in higher dimensions. Specific area of interest would be the Quasi-Monte Carlo method and particularly the Koksma-Hlawka Inequality (See~\cite{nt}).

\section{Acknowledgment}
The authors are to grateful to Andrew Odlyzko, University of Minnesota, USA, for providing valuable insights into the properties of equidistributed sequences; and to Pierre Dusart, University of Limoges, France, for sending paper prints of several publications on prime numbers for reference.

\end{document}